\documentclass[11pt]{amsart}
\usepackage{mathrsfs,amssymb,amsmath}
\usepackage{verbatim}
\usepackage{hyperref}

\usepackage{graphicx}
\begin{document}

\newtheorem{theorem}{Theorem}[section]
\newtheorem{proposition}[theorem]{Proposition}
\newtheorem{lemma}[theorem]{Lemma}
\newtheorem{corollary}[theorem]{Corollary}
\newtheorem{conjecture}[theorem]{Conjecture}
\newtheorem{question}[theorem]{Question}
\newtheorem{problem}[theorem]{Problem}

\theoremstyle{definition}
\newtheorem{definition}[theorem]{Definition}
\newtheorem{example}[theorem]{Example}

\theoremstyle{remark}
\newtheorem{remark}[theorem]{Remark}

\def\theenumi{\roman{enumi}}

\numberwithin{equation}{section}

\renewcommand{\Re}{\operatorname{Re}}
\renewcommand{\Im}{\operatorname{Im}}

\def \R {{\mathbb R}}
\def \HH {{\mathbb H}}
\def \C {{\mathbb C}}
\def \Z {{\mathbb Z}}
\def \Q {{\mathbb Q}}
\def \TT {{\mathbb T}}
\newcommand{\T}{\mathbb T}
\def \Dc {{\mathcal D}}

\newcommand{\tr}[1] {\hbox{tr}\left( #1\right)}

\newcommand{\area}{\operatorname{area}}

\newcommand{\Norm}{\mathcal N}
\newcommand{\simgeq}{\gtrsim}%
\newcommand{\simleq}{\lesssim}

\newcommand{\length}{\operatorname{length}}

\newcommand{\curve}{\mathcal C} 
\newcommand{\vE}{\mathcal E} 
\newcommand{\Ec}{\mathcal {E}} 
\newcommand{\Sc}{\mathcal{S}} 

\newcommand{\dist}{\operatorname{dist}}
\newcommand{\supp}{\operatorname{supp}}
\newcommand{\spec}{\operatorname{spec}}
\newcommand{\diam}{\operatorname{diam}}

\newcommand{\Ccap}{\operatorname{Cap}}
\newcommand{\E}{\mathbb E}

\newcommand{\sumstar}{\sideset{}{^\ast}\sum}

\newcommand {\Zc} {\mathcal{Z}} 
\newcommand{\ninumber}{\Zc}

\newcommand{\zeigen}{E} 
\newcommand{\eigen}{m}

\newcommand{\ave}[1]{\left\langle#1\right\rangle} 

\newcommand{\Var}{\operatorname{Var}}
\newcommand{\Prob}{\operatorname{Prob}}

\newcommand{\var}{\operatorname{Var}}
\newcommand{\Cov}{{\rm{Cov}}}
\newcommand{\meas}{\operatorname{meas}}

\newcommand{\leg}[2]{\left( \frac{#1}{#2} \right)}  

\renewcommand{\^}{\widehat}

\newcommand {\Rc} {\mathcal{R}}

\title[Nodal intersections for random waves on the 3-dim torus]
{Nodal intersections for random waves on the 3-dimensional torus}
\author{Z. Rudnick, I. Wigman and N. Yesha}

\address{
School of Mathematical Sciences, Tel Aviv University, Tel Aviv,
Israel} \email{rudnick@post.tau.ac.il}

\address{Department of Mathematics, King's College London, UK}
\email{igor.wigman@kcl.ac.uk}

\address{
School of Mathematical Sciences, Tel Aviv University, Tel Aviv,
Israel} \email{nadavye1@post.tau.ac.il}

\date{\today}

\thanks{The research leading to these results has received funding from the
European Research Council under the European Union's Seventh
Framework Programme (FP7/2007-2013) / ERC grant agreements
n$^{\text{o}}$ 320755 (Z.R.) and  n$^{\text{o}}$ 335141 (I.W.), and from  the Friends of the Institute for Advanced Study (Z.R.).
}

\begin{abstract}
We investigate the number of nodal intersections of random Gaussian
Laplace eigenfunctions on the standard three-dimensional flat torus
 with a fixed smooth reference curve, which has nowhere vanishing
 curvature. The expected intersection number is universally proportional to the
length of the reference curve, times the wavenumber, independent of
the geometry. Our main result gives a bound for the variance, if
either the torsion of the curve is nowhere zero or if the curve is
planar.
\end{abstract}

\maketitle

\section{Introduction}

\subsection{Toral nodal intersections}

Let $\T^d=\R^d/\Z^d$ be the standard flat $d$-dimensional torus and
$\curve \subset \T^d$ a fixed reference curve\footnote{By a curve we always mean a parameterized, compact, immersed curve. }. Given a real-valued
eigenfunction $F(x)$  of the Laplacian
\begin{equation*}
-\Delta F=4\pi^2\zeigen  \cdot F \;,
\end{equation*}
we wish to study the number of intersections
\begin{equation*}
\Zc(F):=\#\{x\in \T^d: F(x)=0\}\cap \curve
\end{equation*}
of the nodal set of $F$ with the reference curve $\curve$ as a
function of the corresponding eigenvalue.

In dimension $d=2$, if the curve is smooth and has nowhere vanishing
curvature, then deterministically for {\em every} eigenfunction, the
number of nodal intersections satisfies \cite{BRGAFA, BRNI}
\begin{equation*}
\frac{\zeigen^{1/2}}{(\log \zeigen)^{5/2}} \ll \Zc(F) \ll
\zeigen^{1/2};
\end{equation*}
here and everywhere $f\ll g $ (equivalently $ f=O\left(g\right) $) means that
there exists a constant $ C>0 $ such that $ |f|\le C|g| $.
For the upper bound we require that $\curve$ is real-analytic, in
particular it is shown that $\curve$ is not contained in the nodal
set for $\zeigen$ sufficiently large. One can improve the lower
bound to $\Zc(F)\gg \zeigen^{1/2}$ conditionally on a certain
number-theoretic conjecture \cite{BRNI}.

In this note we deal with dimension $d=3$. If we consider the
intersection of the nodal set with a fixed real-analytic reference
{\em surface} $\Sigma\subset \T^3$ with nowhere-zero Gauss-Kronecker
curvature, then for $\zeigen$ sufficiently large,  $\Sigma$ is not
contained in the nodal set \cite{BRINV}, the length of the
intersection of $\Sigma$ with the nodal set is $\ll \sqrt{\zeigen}$
\cite{BRGAFA}, and the nodal intersection is non-empty
\cite{BRGAFA}. However, for the intersection of the nodal set with a
fixed reference {\em curve} (where we expect only finitely many
points), the following examples indicate that one cannot expect to
have any deterministic bounds on the number of nodal intersections.

\begin{example}
Take the eigenfunctions of the form $$F_k(x_1,x_2,x_3)=\sin(2\pi
kx_1);$$ their nodal surfaces are the planes $\left\{x\in
\T^3:x_1\in \frac 1 {2k} \Z \right\}$ and any curve lying on the
plane $x_1=1/2\pi$ does not intersect these, whereas a curve on the
plane $x_1=0$ is lying inside all of these nodal surfaces. We
observe that the curves in this example are {\em planar}.
\end{example}

\begin{example}
Let $F_{0}\left(x,y\right)$ be an eigenfunction on the
two-dimensional torus with eigenvalue $4\pi^2\zeigen_{0}^{2}$, and
$S_{0}$ a curved segment contained in the nodal set, admitting an
arc-length parameterization $\gamma_{0}: [0,L ]\to S_{0}$, with
curvature $\kappa_0(t)=|\gamma_0''(t)|>0$. For $n\ge 0$ let
$F_{n}\left(x,y,z\right)=F_{0}\left(x,y\right)\cos\left(2\pi
nz\right)$, an eigenfunction on $\T^{3}$ with eigenvalue
$4\pi^2(\zeigen_0^2+n^{2})$. Let $\curve$ be the parametric curve
$\gamma\left(t\right)=\left(\gamma_{0}(\frac t{\sqrt{2}}),\frac
t{\sqrt{2}}\right)$. A computation shows that  the curvature is
$\kappa(t) = \frac 12 \kappa_0(\frac t{\sqrt{2}})>0$, and that the
torsion is $\tau(t) = \pm\frac 1{2 }\kappa_0(\frac t{\sqrt{2}})\neq
0$,
so that $\curve$ is non-planar.  Clearly $\curve $ is contained in
the nodal set of $F_{n}$ for all $n$. Thus even in the non-planar
case, we can have the reference curve $\curve$ contained in the
nodal set for arbitrarily large $\zeigen$.
\end{example}

\begin{question} Does there exist a non-planar curve
$\curve$ with no nodal intersections, $\zeigen$ arbitrarily large?
\end{question}

\subsection{Arithmetic random waves}

As there is no deterministic bound on $\Zc(F)$ in dimension $3$, we
investigate what happens for ``typical" eigenfunctions. Let
$\vE(\zeigen)$ be the set of lattice points lying on a sphere
\begin{equation*}
 \vE(\zeigen) = \{\vec x\in \Z^3: |\vec x|^2=\zeigen\},
\end{equation*}
and let
\begin{equation*}
N= N_\zeigen:=\#\vE(\zeigen).
\end{equation*}

The Laplace spectrum on $\T^3$ is of high multiplicities, with the
dimension of an eigenspace corresponding to an eigenvalue
$4\pi^2\zeigen$ being of size $N_\zeigen\approx \zeigen^{1/2 \pm
o(1)}$. The eigenfunctions corresponding to the eigenvalue
$4\pi^2\zeigen$ are of the form
\begin{equation}
\label{eq:F arith wave def} F(x)=F_{\zeigen}\left( x\right) =
\dfrac{1}{\sqrt{N_\zeigen}} \sum_{\mu\in
\mathcal{E}(E)}{a_{\mu}e^{2\pi i \left <\mu , x\right >}}.
\end{equation}
We consider random Gaussian eigenfunctions (``arithmetic random
waves" \cite{KKW}) by taking the coefficients $ a_\mu $ to be
standard complex Gaussian random variables, independent save for the
relations $ a_{-\mu} = \overline{a_{\mu}}$, making $F$ real-valued.

\subsection{Statement of the main results}

\begin{theorem}\label{thmA}
Let $\curve \subset \T^3$ be a smooth curve of length $L$, with
nowhere zero curvature. Assume further that one of the following
holds:
\begin{enumerate}
\item $\curve$ has nowhere-vanishing torsion;
\item $\curve$ is planar (so that the torsion vanishes identically).
\end{enumerate}
Then for all $\epsilon>0$, as $\zeigen \to \infty$ along integers
$\zeigen\not\equiv 0,4,7\bmod 8$, the number of nodal intersections
satisfies
\begin{equation*}
\lim_{\substack{E\to \infty\\ E\not\equiv 0,4,7 \bmod 8}}
\Prob\left(\left|\frac{\Zc(F)}{\sqrt{\zeigen}}-\frac{2}{\sqrt{3}}L
\right|>\epsilon\right) \to 0.
\end{equation*}
\end{theorem}

Note that the condition $\zeigen\not\equiv 0,4,7\bmod 8$ is natural,
as otherwise $\zeigen=4^a \zeigen'$ with $\zeigen'\not\equiv 7 \bmod
8$ (if $\zeigen'\equiv 7\bmod 8$ then $\zeigen$ is not a sum of
three squares and hence does not yield a Laplace eigenvalue); then
an eigenfunction of eigenvalue $4\pi^2\zeigen$ is necessarily of the
form $F(x)=H(2^a x)$ with $H$ an eigenfunction of eigenvalue
$\zeigen'$. Hence any question on the nodal set of $F$ reduces to the
corresponding question on the nodal set of eigenfunctions with
eigenvalue $E'=E/4^a$ (which may be trivial, e.g. if $E=4^a$).

To prove Theorem~\ref{thmA} we compute the expected value of $\Zc$
with respect to the Gaussian measure defined on the eigenspace as
above to be
\begin{equation*}
\E(\Zc) = \frac{2}{\sqrt{3}}L\sqrt{\zeigen},
\end{equation*}
and give an upper bound for the variance:

\begin{theorem}\label{thmvar}
Let $\curve \subset \T^3$ be a smooth curve, with nowhere-zero
curvature. Assume also that either $ \curve $ has nowhere-vanishing
torsion, or $ \curve $ is planar. Then for $\zeigen\not\equiv
0,4,7\bmod 8$,
\begin{equation}
\label{eq:Var<<eig^-delta} \Var
\left(\frac{\Zc}{\sqrt{\zeigen}}\right) \ll \frac
1{\zeigen^{\delta}}
\end{equation}
for all $\delta<1/3$ in case $\curve$ has nowhere vanishing torsion, and
all $\delta<1/4$ in case $\curve$ is planar.
\end{theorem}

\begin{remark}
\label{rem:variance analytic} If $\curve$ is real-analytic and
non-planar, so that the torsion is not identically zero, but may
vanish at finitely many points, the result
\eqref{eq:Var<<eig^-delta} above is valid with some
$\delta=\delta_\curve>0$, see \S~\ref{sec:thmvar proof}.
\end{remark}

\subsection{Outline of the paper and the key ideas}

We prove the approximate Kac-Rice formula (briefly explained in
\S~\ref{sec:Kac-Rice short}) in \S~\ref{sec:Kac-Rice},
followed by a study of certain oscillatory integrals on the curve in
\S~\ref{sec:oscillary integrals}. The arithmetic heart of the paper
is \S~\ref{sec:second_moment}, where we bound the second moment of
the covariance function and its derivatives, following some
background on the arithmetic of sums of three squares in
\S~\ref{sec:sums of 3 squares}, expanded on in Appendix~\ref{sec:Appendix}.

\vspace{3mm}

A similar result to Theorem \ref{thmvar}
was proved in the two-dimensional case for $\curve
\subset \T^2$ having nowhere-zero curvature \cite{RW2014}. In that
case the authors found that the precise asymptotic behaviour of the
nodal intersections variance is {\em non-universal}, namely
dependent on both the angular distribution of the lattice points $
\mathcal{E}\left (E\right ) $ and the geometry of $\curve$. In the
$3$-dimensional case we were only able to obtain an upper bound
\eqref{eq:Var<<eig^-delta} on the variance, which implies the
``almost-all" statement of Theorem~\ref{thmA}. These two cases
differ both in terms of analytic and arithmetic ingredients; the
arithmetic of ternary quadratic forms differs significantly from
that of binary quadratic forms.

\subsection{An approximate Kac-Rice formula}

\label{sec:Kac-Rice short}

We end this Introduction with a discussion of a key step in our work, the approximate Kac-Rice formula.

By restricting the arithmetic random waves \eqref{eq:F arith wave
def} along $\curve$, the problem of nodal intersections count is
reduced to evaluating the number of zeros (zero crossings) of a
random (nonstationary) Gaussian process. The Kac-Rice formula is a
standard tool or meta-theorem for expressing the (factorial) moments
of the zero crossings number of a Gaussian process precisely in
terms of certain explicit integrals with integrands depending on the
given process. 
It is then easy to evaluate the expected number of zeros
precisely and explicitly via an evaluation of a standard Gaussian
expectation.

For the variance, or the second factorial moment, the validity of
the Kac-Rice formula is a very subtle question with a variety of
sufficient conditions known in the literature (e.g. ~\cite{CL,AW}).
While the classical treatise ~\cite{CL} requires the non-degeneracy
of the (Gaussian) distribution of the values of the given process at
two points together with their derivatives, only the non-degeneracy
of the values distribution (at two points) is required for the more
modern treatment ~\cite{AW}, which, to the best knowledge of the
authors, is the weakest known sufficient condition for the validity
of Kac-Rice. Unfortunately, even this weaker condition may fail in
our case.

In order to treat this situation (\S \ref{sec:Kac-Rice}) we
divide the interval into many small subintervals of length
commensurable to the wavelength $\frac{1}{\sqrt{\zeigen}}$
and decompose the total variance as a sum over pairs of subintervals
of zero number covariances \eqref{eq:var=sum cov ij}. We were able
to prove the validity of Kac-Rice for {\em most} of the pairs of
subintervals that includes all the diagonal pairs (i.e. the variance
of nodal intersections along sufficiently small curves), and bound
the contribution of the other pairs via a simple Cauchy-Schwartz
argument. The price is that along the way we incur an error term,
hence yielding an {\em approximate} Kac-Rice (Proposition
\ref{prop:approx Kac-Rice}) reducing a variance computation to an
estimate  for some moments of the covariance function and its
derivatives; such a strategy was also used in \cite{RW2014, CMW}.
The remaining part of this paper is concerned with proving such an
estimate on the second moment of the covariance function and a
couple of its derivatives.

We finally record that we may omit a certain technical assumption
made in our previous work on the two-dimensional case \cite{RW2014}
by using the more general form of the Kac-Rice formula as in
~\cite{AW} (vs. ~\cite{CL}). Given an integer $m$ expressible as a
sum of $2$ squares we defined the probability measure
$$\tau_{m}=\sum\limits_{\|\lambda\|^{2}=m}\delta_{\lambda/\sqrt{m}}$$ on the unit circle $\mathcal{S}^{1}\subseteq \R^{2}$
supported on all the lattice points $\lambda\in\Z^{2}$ lying on the
centered radius-$\sqrt{m}$ circle in $\R^{2}$ projected to the unit
circle. In \cite[Theorem 1.2]{RW2014} we evaluated the
variance of the number of nodal intersections under the assumption
that the Fourier coefficient $\widehat{\tau_{m}}(4)$ is bounded away
from $\pm 1$ (i.e. $\tau_{m}$ bounded away from the singular
measures $\frac{1}{4}(\delta_{\pm 1}+\delta_{\pm i})$ and its
$\pi/4$-tilted version). Using ~\cite{AW} makes that assumption no
longer necessary.

\section{An approximate Kac-Rice formula}
\label{sec:Kac-Rice}

\subsection{The Kac-Rice premise}

Let $d\geq 3$, and $ \curve\subseteq \T^d $ be a smooth curve. Let $
F $ be the arithmetic random wave
\begin{equation*}
F(x)= \dfrac{1}{\sqrt{N_\zeigen}} \sum_{\mu\in
\mathcal{E}(E)}{a_{\mu}e^{2\pi i \left <\mu , x\right >}}
\end{equation*}
(with the obvious generalization of all the previous notation to
higher dimensions $ d\ge 3 $). We wish to study the number $
\mathcal{Z} (F) $ of nodal intersections of $F$ with $\curve$ by
restricting $ F $ to $ \curve $ as follows.

Let $\gamma:[0,L]\rightarrow \R$ be a unit speed parameterization of
$ \curve $. We restrict $ F $ to $ \curve $ by defining the random
Gaussian process
\begin{equation}
\label{eq:f(t)=F(gamma(t))} f(t)=F \left (\gamma (t)\right )
\end{equation}
on $ [0,L] $. It is then obvious that the nodal intersections number
$\mathcal{Z} (F)$ is equal to the number of the zeros of $ f $. The
Kac-Rice formula (see e.g. \cite{CL}, \cite{AW}) is a standard tool
(meta-theorem) for evaluating the expected number and higher
(factorial) moments of zeros of a ``generic" process: let
$X:I\rightarrow \R$ be a (a.s. $C^{1}$-smooth, say) random Gaussian
process on an interval $I\subseteq \R$, and $\Zc=\Zc_{I;X}$ the
number of zeros of $X$ on $I$. For $m\ge 1$ and distinct points $t_{1},\ldots,
t_{m}\in I$ denote $\varphi_{t_{1},t_{2},\ldots, t_{m}}(u_{1},\ldots
u_{m})$ to be the (Gaussian) probability density function of the
random vector $(X(t_{1}),\ldots X(t_{m}))\in \R^{m}$. Then, under
appropriate assumptions on $X$, the $m$-th factorial moment of $\Zc$
is given by
\begin{equation}
\label{eq:Kac_Rice}
\begin{split}
\mathbb E \left[ \Zc^{[m]} \right] = &\int_{I^m}
K_{m}(t_{1},\ldots,t_{m}) \, \mbox{d}t_{1}\ldots \mbox{d}t_{m},
\end{split}
\end{equation}
where
\[ \Zc^{[m]} := \begin{cases} \Zc(\Zc-1)\cdots (\Zc-m+1) &  1 \le m\le \Zc \\ 0 & \mbox{otherwise}. \end{cases}, \]
and $K_{m}$, given by
\begin{equation}
\label{eq:Km def}
\begin{split}
K_{m}(t_{1},\ldots t_{m}) &=
\varphi_{t_{1},\ldots,t_{m}}(0,\ldots,0)\times\\&\times\E \left[
\left \vert X'(t_1)\dots X'(t_m) \right \vert  \, \Big \vert
X(t_1)=0, \dots, X(t_m)=0 \right],
\end{split}
\end{equation}
is the $m$-th zero-{\em intensity} of $\Zc$. Note that for the
Gaussian case
\begin{equation}
\label{eq:phi(0...0) expr} \varphi_{t_{1},\ldots,t_{m}}(0,\ldots,0)
= \frac{1}{(2\pi)^{m/2}\sqrt{\det{A}}}
\end{equation}
with $A$ the covariance matrix of the values
$(X(t_{1}),\ldots,X(t_{m}))$, provided that $\det{A}\ne 0$, or,
equivalently, that the distribution of $(X(t_{1}),\ldots,X(t_{m}))$
is non-degenerate.

The validity of the meta-theorem \eqref{eq:Kac_Rice} was established
under a number of various scenarios. Originally the result
\eqref{eq:Kac_Rice} was proven to hold ~\cite{CL} provided that for
all distinct points $t_{1},\ldots,
t_{m}\in I$ the distribution of the Gaussian
vector $(X(t_{1}),\ldots, X(t_{m}),X'(t_{1}),\ldots
X'(t_{m}))\in\R^{2m}$ is non-degenerate. This non-degeneracy
condition was relaxed\footnote{This fortunate fact simplifies our
treatment of the approximate Kac-Rice formula below (though it is
possible to work ~\cite{RW2014} with the more restrictive version to
obtain the same results).} ~\cite{AW}, as in the following theorem:

\begin{theorem}[~\cite{AW}, Theorem $6.3$]
\label{thm:Kac-Rice} Let $X:I\rightarrow\R$ be a Gaussian process
having $C^1 $ paths and $m\ge 1$. Assume that for every $m $
pairwise distinct points $ t_1,\ldots ,t_{m}\in I $, the joint
distribution of $ (X(t_1), \dots X(t_m))\in \R^{m} $ is
non-degenerate. Then \eqref{eq:Kac_Rice} holds.
\end{theorem}

The cases $m=1,2$ are of our particular interest (see the following
sections). The main problem is that for $m=2$ even the weaker
non-degeneracy hypothesis in Theorem \ref{thm:Kac-Rice} may not be
satisfied for our process $f$ as in \eqref{eq:f(t)=F(gamma(t))}; to
resolve this issue we will decompose the interval $I=[0,L]$ into
small subintervals, and apply Kac-Rice for each pair of the
subintervals to develop ``approximate Kac-Rice" formula following an
idea from \cite{RW2014} in the two-dimensional case (see
\S~\ref{sec:var=sum covar}).

The covariance function of the centered Gaussian random field $ F $
reads\[ r_F(x,y) := \mathbb{E}\left [F(x)F(y)\right ] =
\dfrac{1}{N_\zeigen} \sum_{\mu\in \mathcal{E}}\cos\left (2\pi \left
<\mu ,y-x\right >\right )\] for $ x,y\in \mathbb{T}^d $. As $F$ is
stationary ($r_F$ depending on $ y-x $ only), we may think of $r_F$
as a function of one variable on $\mathbb{T}^d $. The covariance
function of $ f $ is \[ r(t_1,t_2) = r_f(t_1,t_2) := \mathbb{E}\left
[f(t_1)f(t_2)\right ] = r_F\left (\gamma(t_1)-\gamma(t_2)\right ).
\] Therefore, $ f $ is a centered unit variance Gaussian process
(non-stationary); $ r(t_1,t_2)\ne \pm1 $ if and only if the joint
distribution of $ f(t_1),f(t_2) $ is non-degenerate, so the
probability density $ \varphi_{t_1,t_2} $ of the Gaussian random
vector $ \left ( f(t_1),f(t_2) \right ) $ exists. Denote
\[ r_1:=\frac{\partial r}{\partial t_1},  r_2:=\frac{\partial r}{\partial t_2},  r_{12}:=\frac{\partial ^2 r}{\partial t_1 \partial t_2}, \]and let
\begin{equation}
\label{eq:r_2} \Rc_{2}(\zeigen) := \int\limits_{[0,L]^{2}}
\left(r^{2}+(r_{1}/\sqrt{\zeigen})^2 +
(r_{2}/\sqrt{\zeigen})^{2}+(r_{12}/\zeigen)^2\right) \, \mbox{d}t_1
\mbox{d}t_2
\end{equation}
be the sum of second moments of $r$ and its few normalized
derivatives along $\curve$; we will control the various quantities
via $\Rc_{2}$ (see Proposition \ref{prop:approx Kac-Rice} below).
Later we will show that $\Rc_{2}(\zeigen)$ is decaying with
$\zeigen$ (\S~\ref{sec:second_moment}).

\begin{proposition}[Approximate Kac-Rice formula]
\label{prop:approx Kac-Rice} We have
\[ \Var \left (\frac{\mathcal{Z}}{\sqrt{\zeigen}}\right)  = O( \Rc_{2}(\zeigen))\]
with $\Rc_{2}$ is given by \eqref{eq:r_2}.
\end{proposition}

The rest of this section is dedicated to the proof of Proposition
\ref{prop:approx Kac-Rice}, finally given in
\S~\ref{sec:var=sum covar}, following some preparations.

\subsection{Expectation}

Since $ f $ is a centered unit variance Gaussian process with $ C^1
$ paths (in, particular, the non-degeneracy condition of Theorem
\ref{thm:Kac-Rice} is automatically satisfied), we may use the
Kac-Rice formula \eqref{eq:Kac_Rice}; for $m=1$ it reads
\begin{equation}
\label{eq:E[Z]=int(K1)} \E[\Zc] = \int\limits_{I}K_{1}(t) \, \mbox{d}t
\end{equation}
with \[ K_1(t)= \frac{1}{\sqrt{2 \pi}}\cdot \E\left [ \left
|f'(t)\right | \Big| \,  f(t)=0 \right ] ,\] the ``zero density" of
$f$ (here \eqref{eq:phi(0...0) expr} reads
$\varphi_{t}(0)=\frac{1}{\sqrt{2\pi}}$). Let $ \Gamma $ be the
covariance matrix of $ \left (f(t),f'(t)\right ) $: \[ \Gamma(t)=
\begin{pmatrix}
 r(t,t) & r_1(t,t) \\ r_2(t,t) & r_{12}(t,t)
\end{pmatrix}. \]

\begin{lemma}
For a smooth curve of length $ L $, the expectation of nodal intersections number is given by
\begin{equation*}
\E[\Zc] = L  \frac{2}{\sqrt{d}}\cdot \sqrt{\zeigen}
\end{equation*}
\end{lemma}

\begin{proof}
Since $ f $ is unit variance, it is immediate that for every $ t\in
[0,L] $ we have $ r_1(t,t)=r_2(t,t)=0 $, so  \[
\Gamma(t)=\begin{pmatrix} 1&\\&\alpha
\end{pmatrix}, \] where  $$\alpha := r_{12} (t,t) = -\dot{\gamma}(t)^tH_{r_F}(0,0)\dot{\gamma}(t)$$ and $ H_{r_F}$ is the Hessian of $ r_F $ (see e.g. \cite{RW2014}). Since $ H_{r_F}(0,0) $ is a scalar matrix \cite{RW} \[ H_{r_F}(0,0)= \dfrac{-4}{d}\pi ^2\zeigen\ \cdot I_d \] it follows that $$ \alpha =  \dfrac{4}{d}\pi ^2\zeigen. $$ The distribution of $ f'(t)$ conditional on $  f(t)=0  $ is centered Gaussian with variance $ \alpha $. Recall that for $ X  \sim N(0,\sigma^2) $ we have $ \mathbb E(|X|)= \sigma \sqrt{2/ \pi}$, so
\begin{equation} \label{eq:K_1}
K_1(t)=\dfrac{1}{\pi}\sqrt{\alpha}=
\dfrac{2}{\sqrt{d}}\sqrt{\zeigen}
\end{equation}
independent of $x$, and the statement of the lemma follows upon
substituting \eqref{eq:K_1} into \eqref{eq:E[Z]=int(K1)}.
\end{proof}

\subsection{Variance}

For $m=2$ the Kac-Rice formula \eqref{eq:Kac_Rice} reads
\begin{equation}
\label{eq:Kac_Rice2} \E[\Zc^{2}-\Zc]=\int\limits_{I\times I}
K_{2}(t_{1},t_{2}) \, \mbox{d}t_{1}\mbox{d}t_{2}
\end{equation}
with $K_{2}$, the ``$2$-point correlation function", defined for
$r(t_{1}, t_{2})\ne \pm 1$ as (see \eqref{eq:Km def} and
\eqref{eq:phi(0...0) expr})
\[ K_2(t_1,t_2)=\frac{1}{2\pi \sqrt{1-r^{2}}} \cdot \mathbb{E}\left[\left\vert f'(t_1)\right\vert \cdot \left \vert f'(t_2)\right \vert \Big| \, f(t_1)=f(t_2)=0  \right ], \]
holding (Theorem \ref{thm:Kac-Rice}) provided that for all $t_{1}\ne
t_{2}$ we have $r(t_{1},t_{2}) \ne \pm 1$ (equivalently, the
distribution of $(f(t_{1}),f(t_{2}))$ is non-degenerate).
Equivalently, the zero number variance is given by (cf.
\eqref{eq:E[Z]=int(K1)})
\begin{equation}
\label{eq:Kac_Rice var}
 \Var\left (\mathcal{Z}\right ) =\int\limits_{I\times I}\left ( K_2(t_1,t_2) - K_1(t_1)K_1(t_2) \right ) \, \mbox{d}t_1 \mbox{d} t_2 + \mathbb{E}\left [\mathcal{Z}\right ].
\end{equation}

As it was mentioned above, in our case the assumption that $
r(t_1,t_2)\ne \pm 1$ for all $t_1 \ne t_2 $ may not be satisfied,
and, as explained in \cite{RW2014}, it is easy to construct an
example of a curve where the Kac-Rice formula for the second
factorial moment \eqref{eq:Kac_Rice2} does not hold. To resolve this
situation we will divide the interval $I=[0,L]$ into small
subintervals, and note that the proof of ~\cite{AW} Theorem $6.3$
yields that if $J_{1},J_{2}\subseteq I$ are two {\em disjoint}
subintervals, then (recall that we denoted $\Zc_{J}$ to be the
number of zeros of $f$ on a subinterval $J\subseteq I$)
\begin{equation}
\label{eq:Kac-Rice covar disjoint} \E[\Zc_{J_{1}}\cdot
\Zc_{J_{2}}]=\int\limits_{J_{1}\times J_{2}}
K_{2}(t_{1},t_{2}) \, \mbox{d}t_{1}\mbox{d}t_{2},
\end{equation}
provided that for all $t_{1}\in J_{1}$, $t_{2}\in J_{2}$,
$$r(t_{1},t_{2})\ne \pm 1.$$

The $2$-point correlation function was evaluated ~\cite{RW2014}
explicitly to be
\begin{equation}
\label{eq:K2_explicit} K_{2}(t_{1},t_{2}) = \frac{1}{\pi^{2}
(1-r^{2})^{3/2}}\cdot \mu\cdot
(\sqrt{1-\rho^{2}}+\rho\arcsin{\rho}),
\end{equation}
where
\begin{equation}
\label{eq: mu def} \mu = \mu_{\zeigen}(t_{1},t_{2}) =
\sqrt{\alpha(1-r^{2})-r_{1}^{2}}\cdot \sqrt{\alpha
(1-r^{2})-r_{2}^{2}},
\end{equation}
and
\begin{equation}
\label{eq:rho def} \rho = \rho_{\zeigen}(t_{1},t_{2}) =
\frac{r_{12}(1-r^{2})+rr_{1}r_{2}}{\sqrt{\alpha (1-r^{2})-r_{1}^{2}}
\cdot \sqrt{\alpha (1-r^{2})-r_{2}^{2}}}
\end{equation}
(it follows from the derivation of \eqref{eq:K2_explicit} that
$|\rho|\le 1$).

\subsection{Proof of Proposition \ref{prop:approx Kac-Rice}}
\label{sec:var=sum covar}

Before giving the proof of Proposition \ref{prop:approx Kac-Rice} we
will have to do some preparatory work. To overcome the
above-mentioned obstacle we let $ c_0 $ be a sufficiently small
constant to be chosen below, and decompose the interval $[0,L]$ into
small intervals of length roughly $c_{0}\cdot
\frac{1}{\sqrt{\zeigen}}$ so that we can apply Kac-Rice on the
corresponding diagonal cubes. To be more concrete, let $ k = \lfloor
L \cdot \frac{\sqrt{\zeigen}}{c_0} \rfloor + 1 $ and $ \delta_0 =
\frac{L}{k} $, and divide the interval $ [0,L] $ into the
subintervals $ I_i = \left [ (i-1)\delta_0, i\delta_0 \right ] $
where $ i=1,\dots, k $. Note that $ \delta_0 \asymp
\frac{1}{\sqrt{\zeigen}}. $ With $ \mathcal{Z}_i $ denoting the
number of zeros of $ f $ on $ I_i $ $\left ( i=1,\dots , k\right )$
we have
\begin{equation}
\label{eq:var=sum cov ij} \Var(\Zc) =
\sum\limits_{i,j}\Cov(\Zc_{i},\Zc_{j}).
\end{equation}

Our first goal is to give an upper bound for the individual summands
in \eqref{eq:var=sum cov ij}; to this end we need the following
lemmas, whose proofs are postponed till \S~\ref{sec:aux lem
Tayl proof}:

\begin{lemma}
\label{lem:diag_cubes}
There exists a constant $ c_0 > 0 $ sufficiently small, such that \\
for all $ t_1\ne t_2 \in [0,L] $ with $ \left \vert t_2 - t_1 \right
\vert < c_0/\sqrt{\zeigen} $ we have $$ r(t_1,t_2)\ne\pm 1 .$$
\end{lemma}

\begin{lemma}[Uniform bound on the $2$-point correlation function around the diagonal]
\label{lem:ptwise_bnd} For all $0<|t_{2}-t_{1}|< c_{0}/
\sqrt{\zeigen}$ we have
\begin{equation*}
K_{2}(t_{1},t_{2}) = O(\zeigen).
\end{equation*}
\end{lemma}

\begin{corollary}

\label{col:cov_bound} We have
\begin{equation*}
\Cov(\Zc_{i},\Zc_{j}) = O(1),
\end{equation*}
uniformly for all $i,j$ and $\zeigen$ (the implied constant is
universal).
\end{corollary}

\begin{proof}[Proof of Corollary \ref{col:cov_bound} assuming lemmas \ref{lem:diag_cubes}-\ref{lem:ptwise_bnd}.]
By Lemma \ref{lem:diag_cubes}, $ r(t_1,t_2) \ne \pm 1 $ for all $ t_1 \neq t_2 $ in every
diagonal cube $ I_i^2. $ Hence (Theorem \ref{thm:Kac-Rice} applied
on the interval $I_{i}$ corresponding to a diagonal cube $I_i^2$)
we can apply Kac-Rice \eqref{eq:Kac_Rice var} to compute the
variance of $\Zc_{i}$:
\[\Var\left (\mathcal{Z}_i\right ) =\int\limits_{I_i^2}\left ( K_2(t_1,t_2) - K_1(t_1)K_1(t_2) \right ) \, \mbox{d}t_1 \mbox{d} t_2 + \mathbb{E}\left [\mathcal{Z}_i\right ].
\] By Kac-Rice we have \[ \mathbb{E}\left [ \mathcal{Z}_i \right ] = \int\limits_{I_i}K_1(t) \, \mbox{d}t = \delta_0 \cdot \frac{2}{\sqrt{d}}\sqrt{\zeigen}; \] using Lemma \ref{lem:ptwise_bnd} we conclude that $$ \Var\left (\mathcal{Z}_i\right ) \ll \zeigen\delta_0^2+\sqrt{\zeigen}\delta_0 \ll 1 .$$
This proves the statement of the corollary for $ i=j $; the result
for arbitrary $ i,j $ follows from the above and the Cauchy-Schwartz
inequality
\begin{equation*}
\Cov(\Zc_{i},\Zc_{j}) \le \sqrt{\Var\left (\mathcal{Z}_i\right
)\cdot \Var\left (\mathcal{Z}_j\right )}.
\end{equation*}
\end{proof}

\begin{definition}(Singular and nonsingular cubes.)

\label{def:singular}

\begin{enumerate}
\item Let
$$S_{ij} = I_{i}\times I_{j} = [i\delta_{0}, (i+1)\delta_{0}] \times [j\delta_{0}, (j+1)\delta_{0}]$$ be a cube in $[0,L]^{2}$.
We say that $S_{ij}$ is a {\em singular} if it contains a point
$(t_{1},t_{2})\in S_{ij}$ satisfying $$|r (t_{1},t_{2})| > 1/2. $$

\item The union of all the singular cubes is the singular set
$$B=B_{\zeigen} = \bigcup\limits_{S_{ij}\text{ singular}}S_{ij}.$$

\end{enumerate}

\end{definition}

Note that since $r/\sqrt{E}$ is a Lipschitz function with a
universal constant (independent of $ \zeigen $), if $S_{ij}$ is a
singular cube, then $$|r(t_1,t_2)|>1/4$$ everywhere on $S_{ij}$,
provided that $c_{0}$ is chosen sufficiently small. Using the above
it is easy to obtain the following bound on the number of singular
cubes:

\begin{lemma}\label{lem:3.7}
The number of singular cubes is bounded above by \[ E\cdot
\int\limits_{[0,L]^{2}} r^{2}(t_1,t_2)  \, \mbox{d}t_1
\mbox{d}t_2.\]
\end{lemma}

\begin{proof}
Using the Chebyshev-Markov inequality, we see that \[ \meas \left (
B\right ) \ll \int\limits_{[0,L]^{2}} r^{2}(t_1,t_2)  \, \mbox{d}t_1
\mbox{d}t_2.\] The statement of this lemma follows from the fact
that the volume of each cube is $ \asymp 1/\zeigen $.
\end{proof}

With Lemma \ref{lem:3.7} together with Corollary \ref{col:cov_bound}
it is easy to bound the contribution to \eqref{eq:var=sum cov ij} of
all $(i,j)$ corresponding to singular cubes $S_{ij}$ (i.e.
``singular contribution"), see the proof of Proposition
\ref{prop:approx Kac-Rice} below. Next we will deal with the
nonsingular contribution. Here the Taylor expansion of $K_{2}$ as a
function of $r$ and its scaled derivatives around
$r=r_{1}=r_{2}=r_{12}=0$ (up to the quadratic terms) is valid; it
will yield the following result, whose proof will be given postponed
in \S~\ref{sec:aux lem Tayl proof}.

\begin{lemma}
\label{lem:K2=O(r^2 der)} For $(t_{1},t_{2})$ outside the singular
set we have
\begin{equation}
\label{eq:K2 r2 bnd}
\left|K_{2}(t_{1},t_{2})-K_{1}(t_{1})K_{1}(t_{2}) \right| = E\cdot
O\left( r^{2}+(\frac{r_{1}}{\sqrt{\zeigen}})^2 +
(\frac{r_{2}}{\sqrt{\zeigen}})^{2}+(\frac{r_{12}}{\zeigen})^2
\right).
\end{equation}
\end{lemma}



We are finally in a position to give a proof to the main result of
this section.

\begin{proof}[Proof of Proposition \ref{prop:approx Kac-Rice}]

First, it is easy to bound the total contribution of the singular
set to \eqref{eq:var=sum cov ij} (i.e. all $i,j$ with $S_{ij}$
singular): Lemma~\ref{lem:3.7} and Corollary \ref{col:cov_bound}
imply that it is bounded by
\begin{equation}
\label{eq:nonsing contr} \sum\limits_{(i,j):\: S_{ij}\subseteq
B}\Cov(\Zc_{i},\Zc_{j})= O\left (\zeigen \cdot
\int\limits_{[0,L]^{2}} r^{2}(t_1,t_2)  \, \mbox{d}t_1 \mbox{d}t_2
\right ).
\end{equation}
Next we deal with the indexes $(i,j)$ corresponding to nonsingular
$S_{ij}$.

We observe that, by the definition, for such a nonsingular cube
$S_{ij}$, necessarily for every $(t_{1},t_{2})\in S_{ij}$,
$$r(t_1,t_2)\ne \pm 1.$$ As this is a sufficient condition for the
application of Kac-Rice formula \eqref{eq:Kac-Rice covar disjoint}
for computation of $\Cov(\Zc_{i},\Zc_{j})$, bearing in mind
\eqref{eq:K2 r2 bnd} it yields that for $S_{ij}$ nonsingular (this
in particular implies $i\ne j$),
\begin{equation*}
\begin{split}
\Cov(\Zc_{i},&\Zc_{j}) = \int\limits_{S_{ij}}(K_{2}(t_{1},t_{2})-K_{1}(t_{1})K_{1}(t_{2})) \, \mbox{d}t_1\mbox{d}t_2 \\
&= O\left( E \cdot \int\limits_{S_{ij}} \left(
r^{2}+(r_{1}/\sqrt{\zeigen})^2 +
(r_{2}/\sqrt{\zeigen})^{2}+(r_{12}/\zeigen)^2 \right) \, \mbox{d}t_1
\mbox{d}t_2   \right).
\end{split}
\end{equation*}
Hence the total contribution of the nonsingular set to
\eqref{eq:var=sum cov ij} is $O(\zeigen \cdot \Rc_{2}(\zeigen))$. As
the total contribution of the singular set to \eqref{eq:var=sum cov
ij} was   bounded in \eqref{eq:nonsing contr}, and obviously
$$\int\limits_{[0,L]^{2}} r^{2}(t_1,t_2)  \, \mbox{d}t_1 \mbox{d}t_2 \le \Rc_{2}(\zeigen),$$
this concludes the proof of Proposition \ref{prop:approx Kac-Rice}.
\end{proof}

\subsection{Proofs of the auxiliary lemmas \ref{lem:diag_cubes}, \ref{lem:ptwise_bnd}
and \ref{lem:K2=O(r^2 der)}}

\label{sec:aux lem Tayl proof}

\begin{proof}[Proof of Lemma \ref{lem:diag_cubes}]
For $ t_1\in [0,L] $ fixed, we compute the Taylor expansion of $
r(t_1,t_2) $ around $ t_2=t_1 $. Recall that
\begin{equation*}
r(t_1,t_2) = r_F\left (\gamma(t_1)-\gamma(t_2)\right ) =
\dfrac{1}{N} \sum_{\mu\in \mathcal{E}}\cos\left (2\pi \left <\mu
,\gamma(t_1)- \gamma(t_2)\right > \right).
\end{equation*}
Thus, $ r(t_1,t_1)=1 $, $ r_2(t_1,t_1)=0 $ and $ r_{22}(t_1,t_1) =
\dot{\gamma}(t_1)^tH_{r_F}(0)\dot{\gamma}(t_1) = -\alpha $. Moreover, we
clearly have $ r_{222}(t_1,t_2) = O(E^{3/2}) $, and therefore \[
r(t_1,t_2) = 1 - \frac{\alpha}{2}(t_2-t_1)^2 + O\left ( \left (
\sqrt{E}(t_2-t_1) \right )^3\right ). \] Hence, for $ t_2-t_1 \ll
1/\sqrt{E} $ we have
\begin{equation} \label{eq:r_asymp}
\begin{split}
1-r^2(t_1,t_2) &= \alpha (t_2-t_1)^2 + O\left ( \left ( \sqrt{E}(t_2-t_1) \right )^3\right ) \\
&= \alpha(t_2-t_1)^2 \left ( 1+ O\left ( \sqrt{E}(t_2-t_1) \right
)\right ),
\end{split}
\end{equation}
so there is a constant $ c_0>0 $ sufficiently small, such that $
1-r^2(t_1,t_2) $ is strictly positive for $ 0<|t_2-t_1|<c_0 /
\sqrt{E} $.
\end{proof}

\begin{proof}[Proof of Lemma \ref{lem:ptwise_bnd}]
The function \begin{equation} \label{eq:G_rho} G(\rho):=\frac{2}{\pi
} \left ( \sqrt{1-\rho^2}+\rho\arcsin{\rho} \right )
\end{equation}  satisfies $ \frac{2}{\pi}\le G \le 1 $. Hence, by the explicit form \eqref{eq:K2_explicit} of the $ 2 $-point correlation function $ K_2 $  we obtain that
\[ K_2(t_1,t_2) \ll \sqrt{\dfrac{\left ( \alpha \left ( 1-r^2\right ) - r_1^2 \right )\left ( \alpha \left ( 1-r^2\right ) - r_2^2 \right )}{\left ( 1- r^2\right )^3}}. \]
For $ t_2 - t_1 \ll 1/\sqrt{E} $ we have
\[ r_1(t_1,t_2) = \alpha (t_2-t_1)\left ( 1+O\left ( \sqrt{E}(t_2-t_1) \right ) \right ), \]
\[ r_2(t_1,t_2) = -\alpha (t_2-t_1)\left ( 1+O\left ( \sqrt{E}(t_2-t_1) \right ) \right ). \] Using \eqref{eq:r_asymp} we get that
\begin{equation*}
\begin{split}
\alpha(1-r^2)-r_1^2 &=  \alpha^2(t_2-t_1)^2 \left ( 1+O\left ( \sqrt{E}(t_2-t_1) \right ) \right ) \\
&- \alpha^2(t_2-t_1)^2 \left ( 1+O\left ( \sqrt{E}(t_2-t_1) \right )
\right ) = O\left ( E^{5/2}(t_2-t_1)^3 \right )
\end{split}
\end{equation*}
and likewise
\[ \alpha(1-r^2)-r_2^2 = O\left ( E^{5/2}(t_2-t_1)^3 \right ), \]
so
\[ K_2(t_1,t_2) \ll \dfrac{O\left ( E^{5/2}(t_2-t_1)^3 \right )}{\alpha^{3/2}(t_2-t_1)^3\left (1+O\left ( \sqrt{E}(t_2-t_1) \right )\right )} = O(E), \]assuming $ 0<|t_2-t_1|<c_0/\sqrt{E} $ for a sufficiently small constant $ c_0 > 0.$
\end{proof}

\begin{proof}[Proof of Lemma \ref{lem:K2=O(r^2 der)}]
Recall from \eqref{eq:K2_explicit} that
\begin{equation}
\label{eq:K2_def_2} K_2(t_1,t_2) = \frac{1}{2\pi}\cdot
\dfrac{1}{\left ( 1- r^2 \right )^{3/2}} \cdot G(\rho) \cdot \mu
\end{equation}
where $ G $, $ \mu $ and $ \rho $ are defined respectively in
\eqref{eq:G_rho}, \eqref{eq: mu def} and \eqref{eq:rho def}. Note
that for every $|\rho|\le 1 $,
\[ G(\rho) = \dfrac{2}{\pi} + O\left ( \rho ^2 \right). \]
For $ r $, $ r_1/\sqrt{E} $, $ r_2/\sqrt{E} $, $ r_{12}/E $ small,
we have
\begin{equation*}
\rho =  O\left ( r + r_{12}/E \right ).
\end{equation*} Moreover, since $ |\rho| \le 1 $, this bound holds for every $ (t_1,t_2) $. Thus, for every $ (t_1,t_2) $ we have
\[ G(\rho) = \dfrac{2}{\pi}+ O \left ( r^2 + \left (r_{12}/E \right )^2 \right ). \] For every $(t_1,t_2)$,
\[ \mu = \alpha + E\cdot O \left (r^2 + (r_1/\sqrt{E})^2 + (r_2/\sqrt{E})^2  \right ), \] and for $ r $ bounded away from $ \pm  1$ we have \[ \dfrac{1}{\left ( 1- r^2\right )^{3/2}} = 1 + O\left (r^2\right ). \] Substituting all the expansions in \eqref{eq:K2_def_2}, we get that
\[ K_2(t_1,t_2) = \dfrac{\alpha}{\pi ^ 2} + E \cdot O\left ( r^{2}+(r_{1}/\sqrt{\zeigen})^2 + (r_{2}/\sqrt{\zeigen})^{2}+(r_{12}/\zeigen)^2 \right ).  \]
The statement of the lemma now follows, recalling that by
\eqref{eq:K_1}, $ K_1(t) = \sqrt{ \alpha}/\pi$ for all $ t\in
[0,L]$.
\end{proof}

\section{Oscillatory integrals and curvature}
\label{sec:oscillary integrals}

In this section we investigate certain oscillatory integrals on
curves which arise in our work. A key role is played by the
differential geometry of the curve.

\subsection{Differential geometry of \texorpdfstring{$3$}{3}-dimensional curves}
For a smooth curve in $\R^3$, with arc-length parameterization
$\gamma:[0,L]\to \curve \subset \R^3$, so that $T(t) = \gamma'(t)$
is the unit tangent, the curvature of $\gamma$ at $\gamma(t)$ is
$\kappa(t) = ||\gamma''(t)||$. We  assume that $\kappa(t)$ never
vanishes, so that $\gamma''(t)=\kappa(t) N(t)$ with $N(t)$ the unit
normal, and under the same assumption the {\em torsion} $\tau(t)$ is
$B'(t) = -\tau(t)N(t)$ where $B = T\times N$ is the binormal vector.
The orthonormal basis $\left (T,N,B\right ) $ is called the
Frenet-Serret frame of the curve. Recall the Frenet-Serret formulas
\begin{equation*}
\begin{matrix}
T'(t)  = &            & \kappa(t) N(t)&  \\ \\
N'(t) =& -\kappa(t)T(t)&              &+\tau(t)B(t)\\   \\
B'(t) = &&-\tau(t) N(t) &
\end{matrix}
\end{equation*}
so in particular
$$
T''\left(t\right)=
\kappa'\left(t\right)N\left(t\right)-\kappa^{2}\left(t\right)T\left(t\right)+\kappa\left(t\right)\tau\left(t\right)B\left(t\right).
$$

Let $ K_{\min}$ and $K_{\max}$ the minimal and the maximal curvature of $\mathcal{C} $ respectively.
Since the curvature is assumed to be nowhere vanishing, we have
$$0<K_{\min}\le\kappa\left(t\right)\le K_{\max}.$$

\subsection{Oscillatory integrals}


Recall the classical form of Van der Corput Lemma: let
$\left[a,b\right]$ be a finite interval, $\phi\in
C^{\infty}\left[a,b\right]$ a smooth and real valued phase function,
and $A\in C^{\infty}\left[a,b\right]$ a smooth amplitude. For
$\lambda>0$ define the oscillatory integral
\[
I\left(\lambda\right):=\int_{a}^{b}A\left(t\right)e^{i\lambda\phi\left(t\right)}\mbox{d}t.
\]

\begin{lemma}
\label{lem:Van_Der_Corp}(Van der Corput) For $k\ge2$, if
$\left|\phi^{\left(k\right)}\right|\ge1$, then, as
$\lambda\rightarrow\infty$,
\[
\left|I\left(\lambda\right)\right|\ll\frac{1}{\lambda^{1/k}}\left(\left\Vert
A\right\Vert _{\infty}+\left\Vert A'\right\Vert _{1}\right).
\]
If $\left|\phi'\right|\ge1$ and $\phi'$ is monotone, then
\[
\left|I\left(\lambda\right)\right|\ll\frac{1}{\lambda}\left(\left\Vert
A\right\Vert _{\infty}+\left\Vert A'\right\Vert _{1}\right).
\]
The implied constants are absolute.\end{lemma}
\begin{remark}
If $\left|\phi'\right|\ge 1$ then, independent of the monotonicity
hypothesis on $\phi'$,
\[
\left|I\left(\lambda\right)\right|\ll\frac{b-a+2}{\lambda}\left(\left\Vert
A\right\Vert _{\infty}+\left\Vert A'\right\Vert _{1}\right).
\]

\end{remark}

\subsection{Curves with nowhere vanishing torsion}
\label{sec:nonzero_torsion} Assume that the curve $ \curve $ has
nowhere vanishing torsion, so that
$0<T_{\min}\le|\tau\left(t\right)|\le T_{\max}$, where $ T_{\min}$ and
$T_{\max}$ are the minimal and maximal absolute value of the torsion of $\mathcal{C}$
respectively. Consider a unit vector $\xi\in  S^2$, and the phase
function
\begin{equation}
\label{eq:phi phase def} \phi_{\xi}\left(t\right):=\left\langle
\xi,\gamma\left(t\right)\right\rangle
\end{equation}
for $t\in \left [0,L\right ]$. We define an oscillatory integral
\[
I\left(\lambda,\xi\right):=\int_{0}^{L}A\left (t\right
)e^{i\lambda\phi_{\xi}\left(t\right)}\mbox{d}t.
\]
 We apply Lemma \ref{lem:Van_Der_Corp}
to give an upper bound (uniform in $\xi$) for
$I\left(\lambda,\xi\right)$:
\begin{proposition}\label{stationary phase 3d}
Let $\curve$ be a smooth curve with nowhere vanishing curvature and
torsion. Then
$$I\left(\lambda,\xi\right)\ll_{\curve}\frac{1}{\lambda^{1/3}}\left ( \left \Vert A\right\Vert_\infty  + \left \Vert A' \right\Vert _1 \right ).$$
\end{proposition}
\begin{proof}

We have
\begin{equation}
\label{eq:phi'=<xi,T>} \phi'_{\xi}\left(t\right)=\left\langle
\xi,T\left(t\right)\right\rangle,
\end{equation}
\begin{equation}
\label{eq:phi''=k<xi,N>}
\phi_{\xi}''\left(t\right)=\kappa\left(t\right)\left\langle
\xi,N\left(t\right)\right\rangle,
\end{equation}
and
\[
\phi_{\xi}'''\left(t\right)=\kappa'\left(t\right)\left\langle
\xi,N\left(t\right)\right\rangle
-\kappa^{2}\left(t\right)\left\langle
\xi,T\left(t\right)\right\rangle
+\kappa\left(t\right)\tau\left(t\right)\left\langle
\xi,B\left(t\right)\right\rangle .
\]
Since $\left(T,N,B\right)$ is an orthonormal basis for
$\mathbb{R}^{3}$, we know that
\begin{equation}
\label{eq:1=T,N,B decomp}
\begin{split}
1 & =\left|\xi\right|^{2}=\left|\left\langle
\xi,T\left(t\right)\right\rangle \right|^{2}+\left|\left\langle
\xi,N\left(t\right)\right\rangle \right|^{2}+\left|\left\langle
\xi,B\left(t\right)\right\rangle \right|^{2}.
\end{split}
\end{equation}

Now let
\[
c=\min\left(\frac{1}{3},\frac{K_{\min}^{2}T_{\min}^{2}}{48\left\Vert
\kappa'\right\Vert
_{\infty}^{2}},\frac{K_{\min}^{2}T_{\min}^{2}}{48K_{\max}^{4}}\right)
\]
(if $\left\Vert \kappa'\right\Vert _{\infty}=0$ omit the middle
term). If $\left|\left\langle \xi,T\left(t\right)\right\rangle
\right|^{2}\ge c$, then
$\left|\phi_{\xi}'\left(t\right)\right|\ge\sqrt{c}$. If
$\left|\left\langle \xi,N\left(t\right)\right\rangle \right|^{2}\ge
c$, then
$\left|\phi_{\xi}''\left(t\right)\right|\ge\sqrt{c}K_{\min}$.
Otherwise, i.e. if both $\left|\left\langle
\xi,T\left(t\right)\right\rangle \right|^{2}<c$ and
$\left|\left\langle \xi,N\left(t\right)\right\rangle \right|^{2}<c$,
then necessarily $\left|\left\langle
\xi,B\left(t\right)\right\rangle \right|^{2}\ge\frac{1}{3}$, so
\[
\left|\phi_{\xi}'''\left(t\right)\right|\ge\frac{K_{\min}T_{\min}}{\sqrt{3}}-K_{\max}^{2}\sqrt{c}-\left\Vert
\kappa'\right\Vert
_{\infty}\sqrt{c}\ge\frac{K_{\min}T_{\min}}{2\sqrt{3}}.
\]
Note that $\left\Vert \phi_{\xi}'\right\Vert _{\infty}\le1$,
$\left\Vert \phi_{\xi}''\right\Vert _{\infty}\le K_{\max}$,
\[
\left\Vert \phi_{\xi}'''\right\Vert _{\infty}\le\left(\left\Vert
\kappa'\right\Vert
_{\infty}^{2}+K_{\max}^{4}+K_{\max}^{2}T_{\max}^{2}\right)^{1/2}.
\]
Using again the Frenet-Serret formulas, we can also get an upper
bound in the same fashion for the fourth derivative, say
\[
\left\Vert \phi_{\xi}^{\left(4\right)}\right\Vert_\infty \le
C=C\left(K_{\max},T_{\max},\left\Vert \kappa'\right\Vert
_{\infty},\left\Vert \kappa''\right\Vert _{\infty},\left\Vert
\tau'\right\Vert _{\infty}\right).
\]
Assume now that
$\left|\phi_{\xi}'\left(t_{0}\right)\right|\ge\sqrt{c}$ for some
$t_{0}\in\left[0,L\right]$. Then for every $t$ such that
$\left|t-t_{0}\right|\le\frac{\sqrt{c}}{2K_{\max}}$ we have
\[
\sqrt{c}-\left|\phi_{\xi}'\left(t\right)\right|\le\left|\phi_{\xi}'\left(t_{0}\right)\right|-\left|\phi_{\xi}'\left(t\right)\right|\le\left|\phi_{\xi}'\left(t\right)-\phi_{\xi}'\left(t_{0}\right)\right|\le\left|t-t_{0}\right|\left\Vert
\phi_{\xi}''\right\Vert _{\infty}\le\frac{\sqrt{c}}{2}
\]
so $\left|\phi_{\xi}'\left(t\right)\right|\ge\frac{\sqrt{c}}{2}$.
Similarly, if
$\left|\phi_{\xi}''\left(t_{0}\right)\right|\ge\sqrt{c}K_{\min}$ or
$\left|\phi_{\xi}'''\left(t_{0}\right)\right|\ge\frac{K_{\min}T_{\min}}{2\sqrt{3}}$,
then $\phi_{\xi}''$ or $\phi_{\xi}'''$ is bounded away from zero on
some interval around $t_{0}$, with length independent of $\xi$.
Hence the interval $\left[0,L\right]$ may be divided into a finite,
independent of $\xi$, number of subintervals, such that for every
$\xi$ either $\phi_{\xi}'$, or $\phi_{\xi}''$, or $\phi_{\xi}'''$ is
bounded away from zero on each of the subintervals. We conclude the
proof of the proposition by an application of Lemma
\ref{lem:Van_Der_Corp} and the remark following it.
\end{proof}

\subsection{Real analytic curves}

Assume now that $\mathcal{C}$ is a real analytic, non-planar curve
with nowhere zero curvature. Then the torsion of $\mathcal{C}$ has
finitely many zeros, each of them is of finite order. We have
already treated the case when the torsion is nowhere zero. Assume
now, without loss of generality, that there is exactly one point
$t_{0}\in\left[0,L\right]$ with zero torsion of order $m\ge1$,
namely
\begin{equation}\label{eq:torsion_zero}
\tau\left(t_{0}\right)=\dots=\tau^{\left(m-1\right)}\left(t_{0}\right)=0,\hspace{1em}\tau^{\left(m\right)}\left(t_{0}\right)\ne
0.
\end{equation}
Recall that under the notation of the previous section \[
I\left(\lambda,\xi\right):=\int_{0}^{L}A\left (t\right
)e^{i\lambda\phi_{\xi}\left(t\right)}\mbox{d}t.
\]
We prove the following result.
\begin{proposition}
\label{prop:non_plane_analytic} Let $\curve$ be a non-planar real
analytic curve with nowhere zero curvature, which has exactly one
point with zero torsion of order $ m \ge 1 $. Then
$$I\left(\lambda,\xi\right)\ll_{\curve}\frac{1}{\lambda^{1/(m+3)}}\left ( \left \Vert A\right\Vert_\infty  + \left \Vert A' \right\Vert _1 \right ).$$\end{proposition}
\begin{proof}
Using the Frenet-Serret formulas and (\ref{eq:torsion_zero}), we get
that
\begin{alignat*}{1}
\phi_{\xi}^{\left(m+3\right)}\left(t_{0}\right) & =P\left(t_{0}\right)\left\langle \xi,T\left(t_{0}\right)\right\rangle +Q\left(t_{0}\right)\left\langle \xi,N\left(t_{0}\right)\right\rangle \\
 & +\kappa\left(t_{0}\right)\tau^{\left(m\right)}\left(t_{0}\right)\left\langle \xi,B\left(t_{0}\right)\right\rangle
\end{alignat*}
where $P,Q$ are polynomials in
$\kappa,\kappa',\dots,\kappa^{\left(m+1\right)}$.

Choose
\[
c=\min\left(\frac{1}{3},\frac{\kappa\left(t_{0}\right)^{2}\left(\tau^{\left(m\right)}\left(t_{0}\right)\right)^{2}}{48P\left(t_{0}\right)^{2}},\frac{\kappa\left(t_{0}\right)^{2}\left(\tau^{\left(m\right)}\left(t_{0}\right)\right)^{2}}{48Q\left(t_{0}\right)^{2}}\right)
\]
(if $P\left(t_{0}\right)=0$ or $Q\left(t_{0}\right)=0$, omit the
corresponding terms). As in the proof of Proposition \ref{stationary
phase 3d}, by the orthonormality of $\left(T,N,B\right)$, either
$\left|\phi_{\xi}'\left(t_{0}\right)\right|\ge\sqrt{c}$,
$\left|\phi_{\xi}''\left(t_{0}\right)\right|\ge\kappa\left(t_{0}\right)\sqrt{c}$,
or (if both
$\phi_{\xi}'\left(t_{0}\right),\phi_{\xi}''\left(t_{0}\right)$ are
small) $\left|\left\langle \xi,B\left(t_{0}\right)\right\rangle
\right|^{2}\ge\frac{1}{3}$. Hence
\begin{alignat*}{1}
\left|\phi_{\xi}^{\left(m+3\right)}\left(t_{0}\right)\right| & \ge\frac{1}{\sqrt{3}}\kappa\left(t_{0}\right)|\tau^{\left(m\right)}\left(t_{0}\right)|-P\left(t_{0}\right)\sqrt{c}-Q\left(t_{0}\right)\sqrt{c}\\
 & \ge\frac{1}{2\sqrt{3}}\kappa\left(t_{0}\right)|\tau^{\left(m\right)}\left(t_{0}\right)|.
\end{alignat*}
Since all the derivatives of $\phi_{\xi}$ are bounded from above,
uniformly w.r.t. $\xi$, we conclude that either the first, the
second or the $(m+3)$-th derivative of $\phi_{\xi}$ is bounded away
from zero on an interval around $t_{0}$ of length independent of
$\xi$. Outside that interval the torsion doesn't vanish, so that in
a neighborhood (of length independent of $\xi)$ around any point
outside this interval, either the first, or the second, or the third
derivative of $\phi_{\xi}$ is bounded away from zero. Dividing the
interval $\left[0,L\right]$ to a finite number (independent of
$\xi$) of subintervals, and applying Lemma \ref{lem:Van_Der_Corp} to
each of the subintervals, we finally deduce the statement of
Proposition \ref{prop:non_plane_analytic}.
\end{proof}

\section{Background on sums of three squares}
\label{sec:sums of 3 squares}

A positive integer $\zeigen$ is a sum of three squares if and only
if $\zeigen\neq 4^a(8b+7)$. Let $ \vE(\zeigen)$ be the set of
solutions
\begin{equation*}
 \vE(\zeigen) = \{\vec x\in \Z^3: |\vec x|^2=\zeigen\}
\end{equation*}
and denote by $N=N_\zeigen$ the number of solutions
\begin{equation*}
N= N_\zeigen:=\#\vE(\zeigen) \;.
\end{equation*}
Gauss' formula expresses $N_\zeigen$ in terms of class numbers. For
$\zeigen$ square-free, it says that
\begin{equation*}\label{Gauss formula}
\#\vE(\zeigen) = \frac{24
h(d_\zeigen)}{w_\zeigen}\left(1-\leg{d_\zeigen}{2}\right)
\end{equation*}
where $d_\zeigen$, $h(d_\zeigen$) and $w_\zeigen$ are the
discriminant, class number and the number of units in the quadratic
field $\Q(\sqrt{-\zeigen})$. Using Dirichlet's class number formula,
one may then express $\# \vE(\zeigen)$ by means of the special value
$L(1,\chi_{d_\zeigen})$ of the associated quadratic $L$-function: If
$\zeigen\not\equiv 7\bmod 8$ is square-free then
\begin{equation*}\label{N in terms of L(1,chi)}
 N_\zeigen=c_\zeigen\sqrt{\zeigen}\cdot L(1,\chi_{d_\zeigen})
\end{equation*}
where $c_\zeigen$ only depends on the remainder of $\zeigen$ modulo
$8$. We may bound the number $\#\vE(\zeigen)$ of such points as
\begin{equation*}
 \#\vE(\zeigen)\ll\ \zeigen^{1/2+\epsilon}
\end{equation*}
for all $\epsilon>0$.

The existence of a {\em primitive} lattice point (i.e. $\vec
x=(x_1,x_2,x_3)$ with $\gcd(x_1,x_2,x_3)=1$ and $\|x\|^{2}=\zeigen$)
is equivalent to $\zeigen\not\equiv 0, 4,7 \bmod 8$. If it is indeed
the case, then Siegel's theorem yields a lower bound
\begin{equation}\label{Siegel thm}
 \#\vE(\zeigen)\gg \zeigen^{1/2-\epsilon}.
\end{equation}
A fundamental result conjectured by Linnik (established by himself
under the Generalized Riemann Hypothesis), is that for
$\zeigen\not\equiv 0,4,7 \bmod 8$, the points
\begin{equation*}
 \^\vE(\zeigen):= \frac 1{\sqrt{\zeigen}}\vE(\zeigen) \subset S^2
\end{equation*}
obtained by projecting to the unit sphere, become equidistributed on
the unit sphere with respect to the normalized Lebesgue measure as
$\zeigen\to \infty$. This was proved unconditionally by Duke
\cite{Duke}, and Golubeva and Fomenko \cite{GF}.

The ``Riesz $s$-energy" of $N$ points $x_1,\dots, x_N$ on $S^2$ is
defined as
\begin{equation}\label{Def of E_s}
   E_s(x_1,\dots, x_N):=  \sum_{i\neq j} \frac 1{|x_i-x_j|^s}.
\end{equation}
A forthcoming result of Bourgain, Rudnick and Sarnak \cite{BRSsphere} (announced in \cite{BRS} for the electrostatic case $s=1$)
yields a precise asymptotic expression for $E_s(\^\vE(\zeigen))$: for every $0<s<2$ if $\zeigen\to \infty$ such that $\zeigen\neq 0,4,7 \mod 8$,
then there exists some $\delta>0$ so that
\begin{equation}
\label{eq:extremal energy}
E_s(\^\vE(\zeigen)) = I(s) N^2 +O(N^{2-\delta})
\end{equation}
with
\begin{equation*}
I(s)=\int_{S^2} \frac 1{|x - x_0|^s} d\sigma(x) = \frac {2^{1-s}}{2-s},
\end{equation*}
with $x_0\in S^2$ any point on the sphere, and $d\sigma$ the Lebesgue measure, normalized to have unit area

As the details of \eqref{eq:extremal energy} have not appeared at the time of writing, we will prove the following simple bound, which suffices for Theorem \ref{thmvar}:
\begin{proposition}
\label{prop:bd on Riesz2}
Fix $0<s\leq 1$. Then for $\zeigen\not\equiv 0, 4,7 \bmod 8$,
\begin{equation*}
 E_s(\^\vE(\zeigen))  \ll N^2\zeigen^\eta,\quad \forall \eta>0\;.
\end{equation*}
\end{proposition}
The proof of Proposition~\ref{prop:bd on Riesz2} will be given in Appendix~\ref{sec:Appendix}.


\section{The second moment of \texorpdfstring{$r$}{r} and its derivatives}
\label{sec:second_moment}

We wish to bound the second moment of the covariance function $r$
and its derivatives. It is here that we need the full arithmetic
input described in \S~\ref{sec:sums of 3 squares}. Recall that
\begin{equation}
\begin{split} \label{eq:cov_def_2}
r(t_1,t_2) &= r_F\left (\gamma(t_1)-\gamma(t_2)\right ) = \dfrac{1}{N} \sum_{\mu\in \mathcal{E}}\cos\left (2\pi \left <\mu ,\gamma(t_1)- \gamma(t_2)\right >\right ) \\
&= \dfrac{1}{N} \sum_{\mu\in \mathcal{E}}e\left (\left <\mu
,\gamma(t_1)- \gamma(t_2)\right >\right ).
\end{split}
\end{equation}

\subsection{Non-planar curves}

Recall that given $\zeigen$ we defined $\Rc_{2}(\zeigen)$ as in
\eqref{eq:r_2}. Proposition \ref{prop:approx Kac-Rice} shows that in
order to bound the nodal intersections variance from above it is
sufficient to bound $\Rc_{2}(\zeigen)$, which is claimed in the
following proposition for the non-planar case.

\begin{proposition}
\label{prop:R2<<1/E^1/3} Assume that the curve $\curve$ is smooth,
with nowhere zero curvature and torsion. Then for every $\eta>0$ we have
\begin{equation}
\label{eq:Rc2<<1/E^1/3}
\Rc_2(\zeigen) \ll \dfrac{1}{E^{1/3-\eta}}.
\end{equation}
\end{proposition}

\begin{remark}
For real-analytic non-planar curves with non-vanishing curvature,
the same argument as below, invoking Proposition
\ref{prop:non_plane_analytic} instead of Proposition \ref{stationary
phase 3d}, yields
\begin{equation}
\label{eq:Rc2<<1/E^delta analytic} \Rc_2 (\zeigen)\ll 1/E^\delta
\end{equation}
for some $\delta=\delta(\curve)>0$.
\end{remark}

\begin{proof}

In what follows we will establish the following bounds on the $2$nd
moment of $r$ and some of its normalized derivatives along $\curve$: for all $\eta>0$
we have
\begin{equation}
\label{eq:cov_equation}
\iint \limits_{[0,L]^2}
r(t_1,t_2)^2\mbox{d}t_1\mbox{d}t_2  \ll \frac 1{\zeigen^{1/3-\eta}},
\end{equation}
\begin{equation}\label{eq:non_planar_ri}
\iint \limits_{[0,L]^2} \left
(r_i(t_1,t_2)/\sqrt{E}\right )^2\mbox{d}t_1\mbox{d}t_2 \ll \frac
1{\zeigen^{1/3-\eta}} \hspace{10pt} (i=1,2),
\end{equation}
and
\begin{equation}
\label{eq:non_planar_r12} \iint \limits_{[0,L]^2} \left
(r_{12}(t_1,t_2)/E\right )^2\mbox{d}t_1\mbox{d}t_2  \ll \frac
1{\zeigen^{1/3-\eta}}.
\end{equation}
The statement \eqref{eq:Rc2<<1/E^1/3} of Proposition
\ref{prop:R2<<1/E^1/3} will follow at once upon substituting
\eqref{eq:cov_equation}, \eqref{eq:non_planar_ri} and
\eqref{eq:non_planar_r12} into the definition \eqref{eq:r_2} of
$\Rc_{2}(\zeigen)$.

First we show \eqref{eq:cov_equation}. Squaring out and integrating
\eqref{eq:cov_def_2}, we find
\begin{equation*}
\begin{split}
\iint \limits_{[0,L]^2} r(t_1,t_2)^2 \, \mbox{d}t_1\mbox{d}t_2 &=
\frac 1{N^2}\sum_{\mu\in \mathcal{E}} \sum_{\mu'\in \mathcal{E}} \,
\, \iint \limits_{[0,L]^2} e(\langle
\mu-\mu',\gamma(t_1)-\gamma(t_2) \rangle) \, \mbox{d}t_1 \mbox{d}t_2 \\
&= \frac{L^2}{N} + \frac 1{N^2} \sum_{ \mu \neq \mu'} \left|
\int_0^L e(\langle\mu-\mu',\gamma(t) \rangle) \, \mbox{d}t
\right|^2.
\end{split}
\end{equation*}
Since $\gamma$ has nowhere vanishing curvature and torsion, we
deduce from Proposition~\ref{stationary phase 3d} that
\begin{equation*}
 \int_0^L e(\langle\mu-\mu',\gamma(t) \rangle) \, \mbox{d}t \ll \frac 1{|\mu-\mu'|^{1/3}},
 \end{equation*}
which yields
\begin{equation*}
\iint \limits_{[0,L]^2} r(t_1,t_2)^2 \, \mbox{d}t_1\mbox{d}t_2 =
\frac {L^2}{N} + O\left(\frac 1{N^2} \sum_{\mu\neq \mu'} \frac
1{|\mu-\mu'|^{2/3}}\right).
\end{equation*}
The summation inside the error term $O(\cdots)$ is $1/E^{1/3}$ times
the ``Riesz $2/3$-energy" of the set of projected lattice points
$\^\vE(E)=\frac 1{\sqrt{E} }\vE(E)\subset S^2$. By
Proposition \ref{prop:bd on Riesz2},
\begin{equation*}
\sum_{ \substack{\^\mu,\^\mu' \in \^\vE\\ \^\mu\neq \^\mu'}} \frac
1{|\^\mu-\^\mu'|^{2/3}} \ll  N^2\cdot E^{\eta},\forall\eta>0,
\end{equation*}
and hence
\begin{equation*}
\frac 1{N^2} \sum_{\mu\neq \mu'} \frac 1{|\mu-\mu'|^{2/3}} = \frac
1{\zeigen^{1/3}} \frac 1{N^2}\sum_{ \substack{\^\mu,\^\mu' \in \^\vE\\
\^\mu\neq \^\mu'}} \frac 1{|\^\mu-\^\mu'|^{2/3}}\ll \frac{1}{\zeigen^{1/3-\eta}},
\end{equation*}
which yields \eqref{eq:cov_equation}.

Now we turn to proving \eqref{eq:non_planar_ri}. We have
\[ \dfrac{1}{2\pi i \sqrt{E}} r_1(t_1,t_2) = \dfrac{1}{N} \sum_{\mu} \left < \dfrac{\mu}{\left \vert \mu \right \vert}, \dot{\gamma}(t_1)  \right> e^{2\pi i \left < \mu , \gamma (t_1) - \gamma (t_2)  \right > }. \]
Denote
\[ A_{\mu,\mu'}(t) = \left< \dfrac{\mu}{\left \vert \mu \right \vert}, \dot{\gamma}(t)\right> \left< \dfrac{\mu'}{\left \vert \mu' \right \vert}, \dot{\gamma}(t)\right>. \]
Then
\begin{equation} \label{eq:ri_expansion}
\begin{split}
\iint \limits_{[0,L]^2} & \left \vert \dfrac{1}{2\pi  \sqrt{E}}
r_1(t_1,t_2) \, \mbox{d}t_1\mbox{d}t_2 \right\vert ^2 =
\dfrac{L}{N^2} \sum_{\mu} \int_0^L A_{\mu,\mu}(t_1) \, \mbox{d}t_1 \\
&+ \dfrac{1}{N^2}\sum_{\substack{\mu,\mu'\in\vE \\
\mu\ne\mu'}}\int_0^L A_{\mu,\mu'}(t_1)e^{2\pi i \left
<\mu-\mu',\gamma(t_1)\right >} \, \mbox{d}t_1 \int_0^L e^{2\pi i
\left <\mu' - \mu,\gamma(t_2)\right >} \, \mbox{d}t_2.
\end{split}
\end{equation}
To evaluate the main term, we use the fact (see \cite[Lemma
2.3]{RW}) that for every $ v\in \mathbb R ^ d $, \[
\dfrac{1}{N}\sum_{\mu\in\vE}\left <\mu,v\right >^2 =
\dfrac{\zeigen}{d} \left \Vert v \right\Vert ^2.
\] Hence, \[ \dfrac{L}{N^2} \sum_{\mu} \int_0^L A_{\mu,\mu}(t_1)\,
\mbox{d}t_1 =  \dfrac{L^2}{dN}.  \] As for the off-diagonal terms,
note that $ \left\Vert A_{\mu,\mu'} \right\Vert_\infty  \le 1$, $
\left\Vert A'_{\mu,\mu'} \right\Vert_\infty  \le 2K_{\max}$, so by
Proposition \ref{stationary phase 3d} each of the two last integrals
in \eqref{eq:ri_expansion} is bounded above by $
1/|\mu-\mu'|^{1/3}$. From here we continue as in the proof of
\eqref{eq:cov_equation} to obtain the estimate
\[ \iint \limits_{[0,L]^2} \left (r_1(t_1,t_2)/\sqrt{E}\right )^2 \, \mbox{d}t_1\mbox{d}t_2 \ll \frac 1{\zeigen^{1/3-\eta}},\,\forall\eta>0 \]
and a similar proof yields the same bound for the (normalized)
second moment of $ r_2 $.

Finally we turn to proving \eqref{eq:non_planar_r12}. We have
\begin{equation*}
\begin{split}
\iint \limits_{[0,L]^2} & \left \vert \dfrac{1}{4\pi^2 \zeigen}
r_{12}(t_1,t_2) \, \mbox{d}t_1\mbox{d}t_2 \right\vert ^2 =
\dfrac{1}{N^2} \sum_{\mu} \iint \limits_{[0,L]^2} A_{\mu,\mu}(t_1)A_{\mu,\mu}(t_2) \, \mbox{d}t_1\mbox{d}t_2 \\
&+ \dfrac{1}{N^2}\sum_{\substack{\mu,\mu'\in\vE \\
\mu\ne\mu'}}\left\vert \int_0^L A_{\mu,\mu'}(t)e^{2\pi i \left
<\mu-\mu',\gamma(t)\right >} \right\vert ^2 \, \mbox{d}t,
\end{split}.
\end{equation*}
The diagonal term is bounded above by $ L^2/N $; by Proposition
\ref{stationary phase 3d} the off-diagonal terms are bounded above
by  $  1/|\mu-\mu'|^{1/3}$, so we similarly deduce
\eqref{eq:non_planar_r12}.
\end{proof}

\subsection{Planar curves}

The goal of this section is proving the following estimate on
$\Rc_{2}$ for $\curve$ planar.

\begin{proposition}
\label{prop:R2<<1/E^1/4-o(1) planar} Assume that $\curve$ is a
smooth planar curve, with nowhere zero curvature. Then for all $ \eta>0 $
\begin{equation}
\label{eq:R2<<1/E^1/4-o(1) planar} \Rc_2 (\zeigen )\ll
\dfrac{1}{E^{1/4-\eta}}.
\end{equation}
\end{proposition}

We will now collect a few results needed for Proposition
\ref{prop:R2<<1/E^1/4-o(1) planar}, whose proof is given towards the
end of this section. In this section we assume that $\mathcal{C}$ is
a smooth planar curve with nowhere zero curvature, so that
$\tau\equiv0;$ the binormal vector $B$ is constant in this case. Let
$\epsilon=\epsilon\left(E\right)$ be a small parameter,
$\mu\ne\mu'\in\mathcal{E}\left(E\right)$,
$\lambda=\lambda\left(\mu,\mu'\right)=\left|\mu-\mu'\right|$ and
$\xi=\xi\left(\mu,\mu'\right)=\frac{\mu-\mu'}{\left|\mu-\mu'\right|}$.
We will reuse the definition \eqref{eq:phi phase def} of
$\phi_{\xi}$.

\begin{lemma}
\label{lem:small der 2st 2nd} For all $ \xi \in S^2 $ and for all $t\in [0,L]$ either
$\left|\phi_{\xi}'\left(t\right)\right|\ge\sqrt{\epsilon}$, or
$\left|\phi_{\xi}''\left(t\right)\right|\ge\sqrt{\epsilon}K_{\min}$,
or otherwise
$$
\left|\left\langle \xi,B\right\rangle \right|^{2}>1-2\epsilon.
$$
\end{lemma}

\begin{proof}
This follows from \eqref{eq:1=T,N,B decomp} via
\eqref{eq:phi'=<xi,T>} and \eqref{eq:phi''=k<xi,N>}.
\end{proof}

\begin{lemma}
\label{lem:Geom_Lemma} Let
$\mu,\mu_{1},\mu_{2}$ be distinct points on the sphere $ \sqrt{E}S^2 $, and
assume that \[\left\langle
\frac{\mu_{i}-\mu}{\left|\mu_{i}-\mu\right|},B\right\rangle
>\sqrt{1-2\epsilon}\] for $i=1,2$. Then
$\left|\mu_{2}-\mu_{1}\right|\le 16\sqrt{\epsilon E}.$
\end{lemma}

\begin{proof}
Let $v_{i}=\frac{\mu_{i}-\mu}{\left|\mu_{i}-\mu\right|}-B$
$\left(i=1,2\right)$, so that
$\left|v_{i}\right|^{2}=2-2\left\langle
\frac{\mu_{i}-\mu}{\left|\mu_{i}-\mu\right|},B\right\rangle
\le 4\epsilon$. We write
\[
\mu_{i}=\mu+\left|\mu_{i}-\mu\right|\left(B+v_{i}\right).
\]
Taking norms we get
\[
0=\left|\mu_{i}-\mu\right|^{2}+2\left|\mu_{i}-\mu\right|\left\langle
\mu,B+v_{i}\right\rangle,
\]
so that
\[
\left|\mu_{i}-\mu\right|  =-2\left\langle \mu,B\right\rangle -2\left\langle \mu,v_{i}\right\rangle,
\]
and therefore
 \[
 \mu_{i}=\mu-2\left\langle \mu,B\right\rangle
 B-2\left\langle \mu,v_{i}\right\rangle (B + v_i)-2 \left\langle \mu,B\right\rangle v_{i}.
 \]
By Cauchy-Schwartz $\left\langle \mu,v_{i}\right\rangle
\le 2\sqrt{\epsilon E}$, and
$
\left\langle \mu,B\right\rangle \le\sqrt{ E},
$
and that implies $$\left|\mu_{2}-\mu_{1}\right|\le 16\sqrt{\epsilon
E}.$$
\end{proof}

We are now in a position to prove Proposition
\ref{prop:R2<<1/E^1/4-o(1) planar}.

\begin{proof}[Proof of Proposition \ref{prop:R2<<1/E^1/4-o(1) planar}.]
In what follows we will establish the following \\
bounds on the $2$nd
moment of $r$ and some of its normalized derivatives along $\curve$
(assumed to be planar, with nowhere zero curvature):
\begin{equation}
\label{eq:planar_r} \iint \limits_{[0,L]^2} r(t_1,t_2)^2 \,
\mbox{d}t_1\mbox{d}t_2=O\left(\frac{1}{E^{1/4-\eta }}\right),
\end{equation}
\begin{equation}
\label{eq:planar_ri} \iint \limits_{[0,L]^2} \left
(r_i(t_1,t_2)/\sqrt{E}\right )^2\mbox{d}t_1\mbox{d}t_2 = O\left(\frac{1}{E^{1/4-\eta}}\right),
\end{equation}
$i=1,2$, and
\begin{equation}
\label{eq:planar_r12} \iint \limits_{[0,L]^2} \left
(r_{12}(t_1,t_2)/E\right )^2\mbox{d}t_1\mbox{d}t_2  =O\left(\frac{1}{E^{1/4-\eta}}\right).
\end{equation}

First we prove \eqref{eq:planar_r}. If $B$ is the constant binormal
to the curve, then
\begin{equation}
\label{eq:int r^2 << sum I^2}
\begin{split}
\iint \limits_{[0,L]^2} r(t_1,t_2)^2\mbox{d}t_1\mbox{d}t_2= &
\frac{L^2}{N}+\frac{1}{N^{2}}\sum_{\begin{subarray}{c}
\mu\ne\mu'\\
\left|\left\langle\xi\left(\mu,\mu'\right),B\right\rangle
\right|^{2}\le 1-2\epsilon
\end{subarray}}I\left(|\mu-\mu'|,\xi\right)^{2}\\
+ &\frac{1}{N^{2}}\sum_{\begin{subarray}{c}
\mu\ne\mu'\\
\left|\left\langle \xi\left(\mu,\mu'\right),B\right\rangle
\right|^{2}>1-2\epsilon
\end{subarray}}I\left(|\mu-\mu'|,\xi\right)^{2}
\end{split}
\end{equation}
(here $A$, the amplitude involved in $I(\lambda,\xi)$, is
$A(t)\equiv 1$).

To bound the first summation in \eqref{eq:int r^2 << sum I^2} we
observe that for $\mu,\mu'$ with $$\left|\left\langle
\xi\left(\mu,\mu'\right),B\right\rangle \right|^{2}\le 1-2\epsilon$$
we have, thanks to Lemma \ref{lem:small der 2st 2nd}, that for every
$t\in\left[0,L\right]$ either
$\left|\phi_{\xi}'\left(t\right)\right|\ge\sqrt{\epsilon}$ or
$\left|\phi_{\xi}''\left(t\right)\right|\ge\sqrt{\epsilon}K_{\min}$.
Hence Lemma \ref{lem:Van_Der_Corp}, using the same arguments as in
the proof of Proposition \ref{stationary phase 3d}, yields the
following bound, uniform in $\xi$:
\begin{equation}
\begin{split}
I\left(\lambda, \xi
\right)&=\int_{0}^{L}A(t)e^{i\lambda\phi_{\xi}\left(t\right)}\mbox{d}t \\
&\ll_{\mathcal{C}}\min\left\{  \left \Vert A\right\Vert_\infty ,\frac{1}{\left(\sqrt{\epsilon}\lambda\right)^{1/2}}
\left ( \left \Vert A\right\Vert_\infty  + \left \Vert A'
\right\Vert _1 \right ) \right\}.
\end{split}
\end{equation}
We may then bound the first summation in \eqref{eq:int r^2 << sum
I^2} as
\begin{equation}
\label{eq:sum I^{2}<<sum mu,mu'}
\begin{split}
&\frac{1}{N^{2}}\sum_{\begin{subarray}{c}
\mu\ne\mu'\\
\left|\left\langle \xi\left(\mu,\mu'\right),B\right\rangle
\right|^{2}\le 1-2\epsilon
\end{subarray}}I\left(|\mu-\mu'|,\xi\right)^{2}\\
&\ll\frac{1}{N^{2}}\sum_{\begin{subarray}{c}
\mu\ne\mu'\\
|\mu - \mu'|\le 1/\sqrt{\epsilon}
\end{subarray}}1
+\frac{1}{N^{2}}\sum_{\begin{subarray}{c}
	\mu\ne\mu'\\
	|\mu - \mu'|>1/\sqrt{\epsilon}
	\end{subarray}}\frac{1}{\sqrt{\epsilon}\left|\mu-\mu'\right|}\\
&\ll
\frac{1}{N^2}\sum_{\mu}\#\left\{ \mu':|\mu - \mu' | \le 1/\sqrt{\epsilon}\right\} +
\frac{1}{\sqrt{\epsilon}N^{2}} \sum_{
	\mu\ne\mu'}\frac{1}{\left|\mu-\mu'\right|}.
\end{split}
\end{equation}
The second summation in the r.h.s. of \eqref{eq:sum I^{2}<<sum mu,mu'} is
$1/E^{1/2}$ times the ``Riesz $1$-energy" of the set of projected
lattice points $\^\vE(E)=\frac 1{\sqrt{E} }\vE(E)\subset S^2$, which
by Proposition \ref{prop:bd on Riesz2} is bounded by
\begin{equation*}
\sum_{ \substack{\^\mu,\^\mu' \in \^\vE\\ \^\mu\neq \^\mu'}} \frac
1{|\^\mu-\^\mu'|} \ll   N^2\cdot E^{\eta},\, \forall\eta>0,
\end{equation*}
or,
\[
\frac{1}{\sqrt{\epsilon}N^{2}}\sum_{\mu\ne\mu'}\frac{1}{\left|\mu-\mu'\right|}\ll  \frac{E^{-1/2+\eta}}{\sqrt{\epsilon}},\, \forall\eta>0.
\]

For every $ \eta >0, $ the number of lattice points in a spherical cap of radius $r$ on the sphere $RS^2$
can be easily shown to be $ O (R^{\eta}\left(1+r\right))$ (e.g. \cite[Lemma 2.2]{BRGAFA}), so that
the total number of solutions to $|\mu - \mu'| \le 1/\sqrt{\epsilon}$ is
$O\left(E^{\eta }/\sqrt{\epsilon}\right)$.
Hence, the first summation in the r.h.s. of \eqref{eq:sum I^{2}<<sum mu,mu'} is $ O\left ( E^{\eta } /\sqrt{\epsilon}N \right) $, and \eqref{eq:sum I^{2}<<sum mu,mu'} is
\begin{equation}
\label{eq:sum mumu'<1-2eps I^2}
\frac{1}{N^{2}}\sum_{\begin{subarray}{c}
\mu\ne\mu'\\
\left|\left\langle \xi\left(\mu,\mu'\right),B\right\rangle
\right|^{2} \le 1-2\epsilon
\end{subarray}}I\left(|\mu-\mu'|,\xi\right)^{2}\ll\frac{E^{\eta - 1/2}}{\sqrt{\epsilon }}.
\end{equation}

For the second summation on the r.h.s. of \eqref{eq:int r^2 << sum
I^2} we bound each of the integrals $I\left(|\mu-\mu'|,\xi\right)$
from above trivially by $1$, yielding
\[
\begin{split}
&\frac{1}{N^{2}}\sum_{\begin{subarray}{c}
\mu\ne\mu'\\
\left|\left\langle \xi\left(\mu,\mu'\right),B\right\rangle
\right|^{2}>1-2\epsilon
\end{subarray}}I\left(|\mu-\mu'|,\xi\right)^{2}\\
&\ll\frac{1}{N^{2}}\sum_{\mu}\#\left\{ \mu':\left|\left\langle
\xi\left(\mu,\mu'\right),B\right\rangle
\right|^{2}>1-2\epsilon\right\} .
\end{split}
\]
Using Lemma \ref{lem:Geom_Lemma}, we see that all of the lattice
points satisfying \[\left|\left\langle
\xi\left(\mu,\mu'\right),B\right\rangle \right|^{2}>1-2\epsilon\]
are contained in two spherical caps of radius $\ll\sqrt{\epsilon
E}$ so that
the total number of solutions to $\left|\left\langle
\xi\left(\mu,\mu'\right),B\right\rangle \right|^{2}>1-2\epsilon$ is
$O\left(E^{\eta}\left(1+\sqrt{\epsilon E}\right)\right)$.
Therefore
\begin{equation}\label{eq:sum mumu'>1-2eps I^2}
\begin{split}
\frac{1}{N^{2}}\sum_{\begin{subarray}{c}
\mu\ne\mu'\\
\left|\left\langle \xi\left(\mu,\mu'\right),B\right\rangle
\right|^{2}>1-2\epsilon
\end{subarray}}I\left(|\mu-\mu'|,\xi\right)^{2}
&\ll \frac{E^{\eta}\left(1+\sqrt{\epsilon E}\right)}{N} \\
&\ll E^{\eta - 1/2}+\sqrt{\epsilon}E^{\eta}.
\end{split}
\end{equation}

Substituting \eqref{eq:sum mumu'<1-2eps I^2} and \eqref{eq:sum
mumu'>1-2eps I^2} into \eqref{eq:int r^2 << sum I^2} yields the
inequality
\begin{equation}\label{bd 5.11}
\iint \limits_{[0,L]^2} r(t_1,t_2)^2\mbox{d}t_1\mbox{d}t_2=
\frac{L^2}{N}+O\left(\frac{E^{\eta - 1/2}}{\sqrt{\epsilon }}+\sqrt{\epsilon}E^{\eta}\right).
\end{equation}
The estimate \eqref{eq:planar_r} follows by making the optimal
choice $\epsilon=E^{-1/2}$ in \eqref{bd 5.11}.

The proofs of \eqref{eq:planar_ri} and \eqref{eq:planar_r12} are
very similar to the above (cf. the proofs of
\eqref{eq:non_planar_ri} and \eqref{eq:non_planar_r12} within the
proof of Proposition \ref{prop:R2<<1/E^1/3}); we omit the details.
\end{proof}

\begin{remark}
It is possible to slightly improve the exponent of the bound
\eqref{eq:R2<<1/E^1/4-o(1) planar} in Proposition
\ref{prop:R2<<1/E^1/4-o(1) planar} by using a better estimate for
the number of lattice points in spherical caps (\cite[Proposition 1.4]{BRGAFA}).
\end{remark}

\subsection{Concluding the proof of Theorem \ref{thmvar}}

\label{sec:thmvar proof}

\begin{proof}

Use Proposition \ref{prop:approx Kac-Rice} together with either
Proposition \ref{prop:R2<<1/E^1/3} for the\\
nowhere vanishing torsion
case or Proposition \ref{prop:R2<<1/E^1/4-o(1) planar} for the
planar case.
\end{proof}

We finally record that the  proof above will give the result
described in Remark~\ref{rem:variance analytic} for $\curve$
non-planar and analytic, by invoking the bound
\eqref{eq:Rc2<<1/E^delta analytic} 
instead of Proposition \ref{prop:R2<<1/E^1/3}.


\appendix
\section{A simple upper bound for the Riesz energy}\label{sec:Appendix}

We now prove Proposition~\ref{prop:bd on Riesz2}.
Recall that the Riesz energy of the projected lattice points is
$$
E_s(\^\vE(\zeigen))  := \sum_{\mu\neq \nu\in \vE}\frac 1{\left|\frac{\mu}{\sqrt{\zeigen}}-\frac{\nu}{\sqrt{\zeigen}}\right|^s}.
$$
We use a dyadic subdivision to treat the double sum above,
noting that for distinct lattice points $\mu\neq \nu\in \vE$, we have  $1<|\mu-\nu|\leq 2\sqrt{\zeigen}$:
\begin{equation}\label{est riesz 1}
\begin{split}
\sum_{\mu\neq \nu\in \vE}\frac 1{|\mu-\nu|^s}
&=\sum_{1\leq 2^k\leq  2\sqrt{\zeigen} } \; \sum_{2^k\leq |\mu-\nu|<2^{k+1} }  \frac 1{|\mu-\nu|^s}
\\
&\ll  \sum_{2^k \leq 2\sqrt{\zeigen}} \frac 1{2^{ks}} \sum_{\mu\in \vE} \#\{\nu\in \vE: |\mu-\nu|<2^{k+1} \}
\end{split}
\end{equation}
The number of lattice points in a spherical cap of radius $r$ on the sphere $\sqrt{\zeigen}S^2$ is bounded by
$ O (\zeigen^{\eta}\left(1+r\right))$ for all $\eta>0$  \cite[Lemma 2.2]{BRGAFA}, so that
$$ \#\{\nu\in \vE: |\mu-\nu|<2^{k+1} \} \ll \zeigen^\eta 2^k,\quad \forall \eta>0.
$$
Inserting into \eqref{est riesz 1} gives, for $0\leq s\leq 1$, and all $\eta>0$,
\begin{equation*}
\begin{split}
\sum_{\mu\neq \nu\in \vE}\frac 1{|\mu-\nu|^s} & \ll  \sum_{2^k\leq 2\sqrt{\zeigen} }\; \sum_{\mu\in \vE} \frac 1{2^{ks}}2^k E^\eta \\
&=N\zeigen^\eta \sum_{2^k\leq 2 \sqrt{\zeigen}} 2^{k(1-s)} \ll N\zeigen^{(1-s)/2+\eta}\log \zeigen \;.
\end{split}
\end{equation*}
(the factor of $\log\zeigen$ is needed if $s=1$).

Now insert Siegel's lower bound \eqref{Siegel thm} on the number of lattice points:  $\zeigen^{1/2-\delta}\ll N$, for all $\delta>0$, to bound the RHS above by $$ N\zeigen^{(1-s)/2+\eta}\log \zeigen = NE^{1/2-\eta}\cdot E^{-s/2+2\eta}\log{E} \ll
 N^2 \zeigen^{-s/2+3\eta},\,\forall \eta>0.$$
This gives
$$
\sum_{\mu\neq \nu\in \vE}\frac 1{|\mu-\nu|^s} \ll N^2 \zeigen^{-s/2+3\eta},\quad \forall \eta>0,
$$
and hence (replacing $\eta$ by $\eta/3$), for $0<s\leq 1$, the Riesz energy of the projected lattice points is bounded by
$$
E_s(\^\vE(\zeigen))  := \sum_{\mu\neq \nu\in \vE}\frac 1{\left|\frac{\mu}{\sqrt{\zeigen}}-\frac{\nu}{\sqrt{\zeigen}}\right|^s} \ll  N^2 \zeigen^{ \eta},\quad \forall \eta>0 \;.
$$
 This concludes the proof of Proposition~\ref{prop:bd on Riesz2}.

\end{document}